\documentclass[12pt]{amsart}
\usepackage[latin9]{inputenc}
\usepackage{color}
\usepackage{cancel}
\usepackage{amsbsy}
\usepackage{amstext}
\usepackage{amsthm}
\usepackage{amssymb}
\usepackage{mathtools}
\PassOptionsToPackage{normalem}{ulem}
\usepackage{ulem}

\makeatletter
\numberwithin{equation}{section}
\numberwithin{figure}{section}

\usepackage{tikz-cd}
\usepackage{fancyhdr}
\usepackage{latexsym}
\usepackage{amscd}\usepackage{amsfonts}\usepackage{pstricks}
\usepackage{young}
\usepackage{enumitem}
\usepackage{amsthm}
\usepackage{mathrsfs}

\usepackage{ifthen}
\usepackage{longtable}
\usepackage{todonotes}
\usepackage{marginnote}

\renewcommand{\subsection}[1]{\vspace{3mm}\refstepcounter{subsection}\noindent{\bf \thesubsection. #1.} }

\renewcommand{\subsubsection}[1]{\vspace{3mm}\refstepcounter{subsubsection}\noindent{\bf \thesubsubsection. #1.} }

\numberwithin{equation}{section}

\newtheorem{theorem}{Theorem}
\newtheorem{lemma}[theorem]{Lemma}
\newtheorem{corollary}[theorem]{Corollary}
\newtheorem{proposition}[theorem]{Proposition}

\theoremstyle{definition}

\theoremstyle{remark*}
\newtheorem*{remark*}{Remark}

\newtheorem*{example*}{Example}

\newtheorem*{tha}{Theorem A}
\newtheorem*{thb}{Theorem B}

\def\CC{\mathbb C}
\def\PP{\mathbb P}

\DeclareMathOperator{\GL}{GL}

\def\min{\mathop{\mathrm{min}}}
\def\CC{\mathbb C}

\def\PP{\mathbb P}

\newcommand{\bfi}{\mathbf{i}}

\newcommand{\bfx}{\mathbf{x}}

\newcommand{\bfv}{\mathbf{v}}

\newcommand{\bft}{\mathbf{t}}
\newcommand{\exc}{\operatorname{exc}}

\makeatother

\begin{document}
\title[  Defect relation of $n+1$ components through the GCD method]{ Defect of $n+1$ components  through the GCD method}
 
\author{Min Ru}
\address{
 Department  of Mathematics\newline
\indent University of Houston\newline
\indent Houston,  TX 77204, U.S.A.} 
\email{minru@math.uh.edu}
\author{Julie Tzu-Yueh Wang}
\address{Institute of Mathematics, Academia Sinica \newline
\indent No.\ 1, Sec.\ 4, Roosevelt Road\newline
\indent Taipei 10617, Taiwan}
\email{jwang@math.sinica.edu.tw}

\thanks{2020\ {\it Mathematics Subject Classification}: Primary 30D35; Secondary 32Q45 and 32H30}
\thanks{The research of Min Ru is supported in part by Simon Foundations grant award \#531604 and \#521604.
The  second-named author was supported in part by Taiwan's NSTC grant  110-2115-M-001-009-MY3.}

\begin{abstract} This paper studies the defect relation through the GCD method.  In particular, among other results, we extend the defect relation result of   Chen,  Huynh, Sun and  Xie (\cite{CHSX})
to  moving targets. The truncated defect relation is also studied.  Furthermore, we obtain the degeneracy locus, which can be  determined effectively and is independent of the maps under the consideration.
 \end{abstract}

\maketitle
\baselineskip=16truept

 
\section{Motivations}\label{sec:intro}
  
  Let $f: {\Bbb C}\rightarrow {\Bbb P}^n({\Bbb C})$ be a holomorphic map.
   It is known that (see \cite{ru93}, \cite{NW02}, \cite{ru04}) if  $f({\Bbb C})$ omits  $n+2$ smooth hypersurfaces $D_j, 1\leq j\leq n+2$, 
  of ${\Bbb P}^n({\Bbb C})$, where $D:=D_1+\cdots+D_{n+2}$ is located  in general position, then 
  $f$ must be algebraically degenerate (i.e. $f({\Bbb C})$ is contained in a proper sub-variety of  ${\Bbb P}^n({\Bbb C})$). In \cite{NWY07}, J. Noguchi, J. Winkelmann, and K. Yamanoi showed that the number $n+2$ 
  of the omitting hypersurfaces indeed could be reduced to $n+1$ when $\deg D\ge n+2$. 
  Their proof relies on their earlier result for  holomorphic maps  from ${\Bbb C}$ in the semi-abelian variety $A:=({\Bbb C}^*)^{n}$, which is stated as follows.
   \begin{tha}[\cite{NWY08}]\label{Noguchi}  Let $D$ be an effective divisor on $A:=({\Bbb C}^*)^{n}$.
   Let $f : {\Bbb C}\rightarrow A$ be an algebraically nondegenerate holomorphic map. 
   Then there exists a smooth compactification of $A$ independent of $f$, such that
\begin{equation}\label{N1}
N_{f}(D, r)-  N_{f}^{(1)}(D, r)\leq_{\exc} \epsilon\, T_{f, \overline{D}}(r) 
\end{equation}
for any $\epsilon >0$.
\end{tha}

Using the above theorem, Noguchi, Winkelmann, and  Yamanoi showed that one can reduce the number of omitting divisors by one (i.e., from $n+2$ to $n+1$) whose 
argument  is similar to ours described  in Section 5.  We briefly outline the argument here: 
 Assume that $D_j=\{Q_j=0\}$ and $f({\Bbb C})$ omits $D_j$ for $1\leq j\leq n+1$. Consider a morphism $\pi:\PP^{n}(\Bbb C)\to\PP^{n}(\Bbb C)$
 given by
 {${\mathbf x}\mapsto[Q_{1}^{a_{1}}({\mathbf x}):\hdots:Q_{n+1}^{a_{n+1}}({\mathbf x})]$,
where $a_{i}\coloneqq{\rm lcm}(\deg Q_{1},\hdots,\deg Q_{n+1})/\deg Q_{i}$.
  Let 
\begin{align*}
G:=\det \left(\frac{\partial Q_i }{\partial x_j}\right)_{1\le i\le n+1, 0\le j\le n}\in {\Bbb C}[x_{0},\hdots,x_{n}].
 \end{align*} 
 By taking out  a nonconstant irreducible factor $\tilde G$ of $G$  in $K[x_{0},\hdots,x_{n}]$, one produces an additional hypersurface ${\tilde D}_{n+2}=\{\tilde G=0\}$ in $\PP^{n}(\Bbb C)$.
 Furthermore, one can show that  $ D_1, \dots, D_{n+1},$ $\tilde D_{n+2}$ are located in general position, and, by using (\ref{N1}), one can show that $N_f({\tilde D}_{n+2}, r) \leq_{\exc} \epsilon\, T_f(r)$. Thus we can apply the Second Main Theorem obtained by the first author (\cite{ru04}) to get the conclusion.

In a recent manuscript by Z. Chen, D.T. Huynh, R. Sun and S.Y. Xie (see \cite{CHSX}), the result of Noguchi-Winkelmann-Yamanoi mentioned above was further extended to the following  about the defect relation.
\begin{thb}[Z. Chen, D.T. Huynh, R. Sun and S.Y. Xie, \cite{CHSX} ] Let $\{D_i\}_{i=1}^{n+1}$ be $n+1$ hypersurfaces with total degrees $\sum_{i=1}^{n+1} \deg(D_i)\ge n+2$ 
satisfying one precised generic condition $($see (4.3) in their paper  \cite{CHSX}$)$. Then, for every algebraically non-degenerate entire holomorphic curve $f: 
{\Bbb C}\rightarrow {\Bbb P}^n({\Bbb C})$, the following defect relation holds:
$$\sum_{i=1}^{n+1} \delta_f(D_j)<n+1.$$
\end{thb}
Note that, in the omitting case, we have that $Q_j(f)$ is nowhere zero for all $1\leq j\leq n+1$ where $D_j=\{Q_j=0\}$ so that one can reduce it to the semi-abelian  variety case $({\Bbb C}^*)^n$ by 
considering $F:=[Q_1(f), \dots, Q_{n}(f)]\in ({\Bbb C}^*)^n$, assuming that $\deg(Q_1)=\cdots=\deg(Q_{n+1})$. Hence Theorem A could be applied directly. 
However, in their ``proof-by-contradiction" argument of the proof of Theorem B, the condition that $\sum_{i=1}^{n+1} \delta_f(D_j)=n+1$ only implies  that  $N_f(r, D_i)={\rm o}(T_f(r))$ (rather than $f({\Bbb C})$ omitting $D_j$). To overcome this difficulty, they used the ``parabolic Nevanlinna theory"
developed by M. P${\rm{\breve{a}}}$un and N. Sibony (see \cite{paun2014value}), by considering the holomorphic mapping $f: Y\rightarrow {\Bbb P}^n({\Bbb C})$  with $Y:={\Bbb C}\backslash f^{-1}(D)$, which leads to the omitting case after  restricting $f$ to $Y$. The key ingredient in their paper is to show that $Y$ is an open parabolic Riemann surface with exhaustion function $\sigma$ satisfying 
$$\limsup_{r\rightarrow \infty} {\mathfrak{X}_{\sigma}(r)\over T_f(r)}=0.$$
Note that while the method of Chen-Huynh-Sun-Xie is very interesting and creative, it still relies on the result of  Noguchi, Winkelmann, and  Yamanoi (Theorem A), which greatly depends on the geometry of semi-abelian varieties.   For example, it is very hard to generalize the result to the moving target case. 

This paper studies the defect relation through the GCD method.  We don't use Theorem A. Indeed, we give and prove a variant and more general version of Theorem A by using the GCD theorem established by 
Aaron Levin and  the second author \cite{levin2019greatest}.  This allows us to get a much more general defect relation (for example the moving target case).   The method was initiated by P. Corvaja and U. Zannier (see \cite{CZ}), where they studied  the $n=2$ case. 
After Aaron Levin and the second author \cite{levin2019greatest}  established  the general GCD theorem, it has been successfully used in a series papers by the second author and her co-authors, see 
\cite{GNSW}, \cite{GSW20}, \cite{GSW22}, \cite{GW22}.   The purpose of this paper is to  further  use the ideas developed in \cite{GNSW}, \cite{GSW20}, \cite{GSW22}, \cite{GW22} to extend the defect relation, for example to the moving target case, by using the GCD method. Furthermore, the truncated defect relation is also studied.  We also pay attention on the degenerate locus. In particular, we can relax the condition that $f$ being algebraically non-degenerate to the condition that the image of $f$ is not contained in a subvariety $Z$ which can be effectively pre-determined and is independent of $f$,  in the sprite of the strong Green-Griffith-Lang conjecture.

\section{Statement of the results}
   We use the standard notations in Nevanlinna theory (see \cite{ru2001nevanlinna}, or \cite{GNSW}, \cite{GSW20}, \cite{GSW22}, \cite{GW22}). Let $\mathbf{g}=(g_0,\ldots,g_n):\CC\to\PP^n(\CC)$ be a holomorphic curve, where $g_0,\hdots,g_n$ are entire functions without common zero.  We recall that the  {\it small field with respect to   $\mathbf{g}$}  is given by
\begin{equation}\label{smallfield}
    K_{\mathbf{g}}:=\{a:a \text{ is a meromorphic function on ${\Bbb C}$  with } T_{a}(r)= {\rm o}( T_{\mathbf{g}}(r))\}.
\end{equation} 

Let $K$ be a subfield of the field $\mathcal M$ of meromorphic functions.  We say that  {\it $Z$ is a Zariski closed subset in $\mathbb P^n$ defined over $K$} if  there exists a non-constant homogeneous polynomial  $F\in K [x_0,\hdots,x_n]$ such that $Z= \{[f_0:\cdots:f_n]\in \mathbb P^n(\mathcal M): F(f_0,\ldots,f_n)\equiv 0\}$.  We say a  holomorphic map $\mathbf{g} :\CC\to \PP^n$ is not contained in $Z$ if  $\mathbf{g}$ is not identically zero.
In particular, when $K=\mathbb C$, the Zariski closed set is defined over $\mathbb C$, i.e. $F\in \mathbb C [x_0,\hdots,x_n]$, and $\mathbf{g}$ is not contained in $Z$ is equivalent to $F(\mathbf{g})\not\equiv 0$.
 For each polynomial $G=\sum_{I}a_{I}{{\bf x}^{I}}\in K[x_{0},\hdots,x_{n}]$, where
$I=(i_{0},\hdots,i_{n})\in\mathbb{Z}_{\ge0}^{n+1}$ and ${\bf x}^{I}=x_{0}^{i_{0}}\cdots x_{n}^{i_{n}}$, we denote by $G({z_{0}}):=\sum_{I}a_{I}(z_{0}){{\bf x}^{I}}$ if all the coefficients of $G$ are holomorphic at $z_0$ and do not vanish simultaneously at  $z_0$.
Let $G_1,\hdots,G_q$ be non-constant  polynomials  in  $K[x_{0},\hdots,x_{n}]$.  We say that they are {\it in weakly general position} if there exist a point $z_0\in \mathbb C$ such that  for each $G_i({z_{0}})$, $1\le i\le q$, can be defined as above and the zero locus of $G_i({z_{0}})$ (as a divisor in $\mathbb P^n(\mathbb C)$), $1\le i\le q$,  are in general position.
\begin{theorem}\label{main_thm_1GExc}
Let $K$ be a subfield of the field $\mathcal M$ of meromorphic functions.
Let $F$ be a nonconstant  homogeneous polynomial in $K[x_0,\hdots,x_n]$ with no monomial  factors and no repeated factors.  Denote by $H_i$, $0\le i\le n$,   the coordinate hyperplanes of $\mathbb P^n$. 
Then, for any   $\epsilon >0$, there exists a proper Zariski closed subset $Z$ of $\mathbb P^n$ defined over $K$ such that for any nonconstant holomorphic curve   $\mathbf{g}:\CC\to\PP^n(\CC)$ with $K\subset K_{\mathbf{g}}$, $N_{\mathbf{g}}(H_i,r)={\rm o}( T_{\mathbf{g}}(r))$ for $0\le i\le n$ and  $\mathbf g $  is not contained in $Z$, we have   
\begin{align}\label{multizerobdd}
	N_{\mathbf{g}}([F=0],r)-N^{(1)}_{\mathbf{g}}([F=0],r)\le_{\rm exc} \epsilon T_{\mathbf{g}}(r).
\end{align}
If we assume furthermore that  the hypersurface  defined by   $F$ in $\mathbb P^n$ and the coordinate hyperplanes are in weakly general position, then
\begin{align}\label{truncate1_2}
N^{(1)}_{\mathbf{g}}([F=0],r)\ge_{\rm exc}  (\deg  F- \epsilon)\cdot T_{\mathbf g}(r).
\end{align}
Moreover, the exceptional set $Z$ can be expressed as the zero locus of a finite set
$\Sigma\subset K[x_0,\ldots,x_n]$ with the following properties: {\rm(Z1)} $\Sigma$ depends on $\epsilon$ and $F$  only  and can be determined explicitly;  {\rm(Z2)}   the degree of each polynomial in $\Sigma$ can be effectively bounded from above in terms of $\epsilon$, $n$, and the degree of $F$.   
 \end{theorem}
\begin{theorem}\label{GG_conj}
Let $K$ be a subfield of the field of meromorphic functions. Let 
$F_i$, $1\le i\le n+1$, be homogeneous irreducible polynomials of positive degree  in $K[x_0,\hdots,x_n]$ such that $\sum_{i=1}^{n+1}\deg F_i\ge n+2$. 
Assume that there exists $z_0\in \CC$ such that all the coefficients of  all $F_i$, $1\le i\le n+1$, are  holomorphic at $z_0$  and the zero locus of  $F_i$  evaluated at $z_0$,  $1\le i\le n+1$, intersect transversally.  Then there exists a non-trivial homogeneous polynomial $B\in K[x_0,\hdots,x_n]$ such that for any nonconstant holomorphic map $\mathbf{f}:\CC\to \PP^n$  with $K\subset K_{\mathbf{f}}$ and $N_{F_i(\mathbf{f})} (0,r)={\rm o}( T_{\mathbf{f}}(r))$ for $1\le i\le n+1$, we have $B(\mathbf{f})\equiv 0$.  Furthermore, $B$ can be determined effectively and its degree can be effectively bounded from above in terms of  $n$, and the degrees of $F_i$, $1\leq i\leq n+1$.
\end{theorem}

\begin{theorem}[Defect relation for moving targets]\label{defect}Let $K$ be a subfield of the field of meromorphic functions. Let 
$F_i$, $1\le i\le n+1$, be homogeneous irreducible polynomials of positive degree  in $K[x_0,\hdots,x_n]$ such that $\sum_{i=1}^{n+1}\deg F_i\ge n+2$. 
Assume that there exists $z_0\in \CC$ such that all the coefficients of  $F_i$, $1\le i\le n+1$, are  holomorphic at $z_0$  and the zero locus of  $F_i$ evaluated at $z_0$,  $1\le i\le n+1$, intersect transversally.  Then there exists a non-trivial homogeneous polynomial $B\in K[x_0,\hdots,x_n]$ such that for any nonconstant holomorphic map $\mathbf{f}:\CC\to \PP^n$  with $K\subset K_{\mathbf{f}}$ and $B(\mathbf{f})\not\equiv 0$, we have
the following defect inequality holds
$$
\sum_{i=1}^{n+1}\delta_{\mathbf{f}} (D_i)<n+1,
$$
where $D_i=[F_i=0]$.

Additionally, if $n=2$, then $$
\sum_{i=1}^{3}\delta_{\mathbf{f}}^{(1)}(D_i)<3, $$where, for a divisor $D$ with  $d=\deg(D)$, 
 $$\delta_{\mathbf{f}}^{(1)}(D) = 1-\limsup_{r\rightarrow \infty} {N^{(1)}_f(r, D)\over dT_f(r)}.$$
\end{theorem}

Note that when $K={\Bbb C}$, then $B\in {\Bbb C}[x_0,\hdots,x_n]$, so we have the following strong defect relation in the sprite of the strong Green-Griffith-Lang conjecture.
\begin{corollary} Let  $D_i$, $1\le i\le n+1$, be $n+1$ hypersurfaces in $\mathbb P^n(\CC)$, not all being hyperplanes.
Assume $D_i$, $1\le i\le n+1$, intersect transversally.  Then there exists a Zariski closed subset $Z$ in $\mathbb P^n(\CC)$, which 
can be determined effectively and its degree can be effectively bounded from above in terms of $n$, and the degree of $D_i$, 
  such that for any non-constant holomorphic map $\mathbf{f}:\CC\to \PP^n(\CC)$ whose image  is  not  contained in $Z$, the following defect inequality holds
$$
\sum_{i=1}^{n+1}\delta_{\mathbf{f}} (D_i)<n+1.
$$
Additionally, if $n=2$, then $$
\sum_{i=1}^{3}\delta_{\mathbf{f}}^{(1)}(D_i)<3.
$$
\end{corollary}

\section{Some preliminary and the GCD theorem}\label{preliminary} 

\subsection{Preliminary}
We will make use of the following   elementary inequality  (see \cite{ru2001nevanlinna}). 
\begin{proposition} 
	\label{basic_prop}
	Let ${\mathbf f}=(f_0,\hdots,f_n):\mathbb C\to \PP^n(\mathbb C)$ be  holomorphic curve, where $f_0,\dots,f_n$ are entire functions without common zeros.  Assume that $f_0$ is not identically zero.  Then  
	\begin{equation*}
		T_{f_j/f_0}(r)+{\rm O}(1)\leq T_{\mathbf f}(r)\leq \sum_{j=1}^nT_{f_j/f_0}(r)+{\rm O}(1).
	\end{equation*}
\end{proposition}
 
 Combining the above proposition, we state the following result.
\begin{theorem}[{\cite[Theorem 2.1]{RuWang2003}}]\label{trunborel}  
	 Let $f_0,\dots,f_n$ be entire functions with no common zeros.  Assume that $f_{n+1}$ is the holomorphic function such that  $f_0+\dots+f_n+f_{n+1}=0$. If $\sum_{i\in I}f_i\ne 0$ for any proper subset $I\subset\{0,\dots,n+1\}$, then
	\begin{equation*}
		T_{f_j/f_i}(r)\le T_{\mathbf{f}}(r)+{\rm O}(1)\leq_{\exc} \sum_{i=0}^{n+1} N_{f_i}^{(n)}(0,r)+{\rm O}(\log T_{\mathbf{f}}(r)) 
	\end{equation*}
	 for any pair $0\le i,j\le n$, where  $\mathbf{f}:=(f_0,\hdots,f_{n})$.
\end{theorem}

We will need the following version of Hilbert Nullstellensatz reformulated from \cite[Proposition 2.1]{DethloffTan2011}. (See also \cite[Chapter~XI]{waerden1967}.)
\begin{proposition}[{\cite[Proposition 2.1]{DethloffTan2011}}]\label{HilbertN}
  Let $K$ be  a subfield of the field of meromorphic functions.  Let $\{Q_i\}_{i=1}^{n+1}$ be a set of homogeneous polynomials in $ K [x_0,\dots,x_n]$  in weakly general position and with $\deg Q_j=d_j\ge 1$.   Then there exist a positive integer $s$,  an element $R\in K $ which is not identically zero and  $P_{ji}\in K[x_0,\dots,x_n]$, $1\le i,j\le n+1$,  such that, for each $0\le j\le n$, 
$$x_j^s\cdot R=\sum_{i=1}^{n+1} P_{ji} Q_i.
$$
\end{proposition}

The following is a version of the Borel Lemma for small functions.  The proof can  easily be obtained with  some slightly modifications from  \cite[Lemma 3.3]{GW22}.
\begin{lemma} \label{Mborel1}
   Let $f_0,\hdots,f_n$ be nontrivial entire   functions  with no common zero and let $\mathbf{f}:=(f_0,\hdots,f_{n})$. 
  Assume that   
 $$
   N^{(1)}_{f_i}(0,r)={\rm o}( T_{\mathbf{f}}(r)), \quad\text{ for $0\le i\le n$}.
   $$
If $f_0,\hdots,f_n$  are linearly dependent over $K_{\bf f}$, then for each $f_i$, there exists $j\ne i$ such that   $ f_i/f_j \in K_{\bf f}$.
\end{lemma}

\subsection{The GCD theorem}
\begin{theorem}[The GCD theorem]\label{Mgcdunit}
Let $g_0,g_1,\dots,g_n $ be entire functions without common zeros and let ${\bf g}=[g_0:g_1:\cdots:g_n]$.  
Let $F,G\in K_{\bf g}[x_0,\hdots,x_n]$ be nonconstant coprime homogeneous polynomials.    Assume that one of the following holds:
\begin{enumerate}
\item[{\rm(a)}]$N_{g_i}(0,r)={\rm o}(T_{\bf g}(r))$ for $0\le i\le n$;
\item[{\rm(b)}] $N^{(1)}_{g_i}(0,r)={\rm o}(T_{\bf g}(r))$ for $0\le i\le n$ and  one of  the hypersurfaces  defined by   $G=0$ or $F=0$ in $\mathbb P^n(K)$ is in weakly general position with the $n+1$ coordinate hyperplanes. 
\end{enumerate}
  Then, for  any  $\epsilon>0$,
there exist a positive integers $m$  independent of   ${\bf g}$    such that    we have  either
	\begin{align}\label{gcd0}
		N_{\rm gcd} (F(g_0,\hdots,g_n),G(g_0,\hdots,g_n),r)\le_{\rm exc}  
		\epsilon T_{\bf g}(r),  
	\end{align} 
	or 
	\begin{align}\label{gdegenerate}
		T_{(\frac{g_1}{g_0})^{m_1}\cdots (\frac{g_n}{g_0})^{m_n}}(r)= {\rm o} (T_{\bf g}(r))
	\end{align}
	for some non-trivial tuple of integers $(m_1,\hdots,m_n)$ with $|m_1|+\cdots+|m_n|\le 2m$.
\end{theorem}
 For the convenience of later application, we state the following result for $n=1$.
 \begin{proposition}  \label{gcdn1}
Let $g_0,g_1$ be entire functions without common zeros and let ${\bf g}=(g_0,g_1)$.   Assume that ${\bf g}$ is not constant.
Let $F,G\in K_{\bf g}[x_0,x_1]$ be nonconstant coprime homogeneous polynomials.    Then 
	\begin{align}\label{gcd0}
		N_{\rm gcd} (F(g_0,g_1),G(g_0,g_1),r)\le  {\rm o} (T_{\bf g}(r)). 
	\end{align} 
\end{proposition}
\begin{proof}
	Since $F$ and $G$ are coprime homogeneous polynomials in  $K_{\bf g}[x_0,x_1]$, we may apply  Proposition \ref{HilbertN} to find an integer $s$, $R\in K_{\bf g}\setminus\{0\} $ and  $H_i\in  K_{\bf g}[x_0,x_1]$, $1\le i\le 4$, such that 
	\begin{align}\label{resul}
		x_0^s\cdot R=H_1F+H_2G\quad\text{and }\quad   x_1^s\cdot R=H_3F+H_4G. 
	\end{align}
Here, we may assume that $H_i$, $1\le i\le 4$, are homogeneous polynomials with degree equal to $s-\deg F$.
By evaluating \eqref{resul} at $(g_0,g_1)$,  we have
	\begin{align}\label{resulevaluating }
		g_0^s\cdot R &= H_1(g_0,g_1) F(g_0,g_1)+H_2(g_0,g_1) G(g_0,g_1),\cr
		g_1^s\cdot R &=H_3(g_0,g_1) F(g_0,g_1)+H_4(g_0,g_1) G (g_0,g_1). 
	\end{align}
	Since $g_0$ and $g_1$ have no common zeros, we observe that 
	\begin{equation}\label{min_ord}
		\min\{v_z^+ (F(g_0,g_1)), v_z^+(G(g_0,g_1))\} \le v_z^+(R) +\sum_{\alpha\in I}  v_z^-(\alpha)
	\end{equation}
	for each $z\in\mathbb{C}$.  Here $I$ is the set of nontrivial coefficients of $H_i$, $1\le i\le 4$. Hence,		\begin{equation}\label{gcd_GB3}
		N_{\gcd} (F(g_0,g_1), G(g_0,g_1),r)\le N_R(0,r)+ \sum_{\alpha\in I}  N_{\alpha }(\infty, r )\le   {\rm o}(T_{\bf g}(r)), 
	\end{equation}
as $R$ and the coefficients of $F_i$ are in  $K_{\bf g}$. 
\end{proof}

To prove Theorem \ref{Mgcdunit}, we use the following fundamental result by Levin and the second author for $n\ge 2$.
\begin{theorem}[{\cite[Theorem 5.7]{levin2019greatest}}]\label{Mfundamental}
Let $g_0,g_1,\dots,g_n $ be entire functions without common zeros with $n\ge 2$  and let  ${\bf g}=[g_0:g_1:\cdots:g_n]$.  
Let $F,G\in K_{\bf g}[x_0,x_1,\hdots,x_n]$ be coprime homogeneous polynomials of the same degree $d>0$.  Let $I$ be the set of exponents ${\bf i}$ such that ${\bf x}^{\bf i}$ appears with a nonzero coefficient in either $F$ or $G$.  Let $m\ge d$ be a positive integer.   Suppose that the set $\{g_0^{i_0}\dots g_n^{i_n}: i_0+\cdots+i_n=m\}$ is linearly independent over $K_{\bf g}$.   Then, for any $\epsilon>0$, there exists a positive integer $L$ such that the following holds:
\begin{align}\label{fundmenatlgcd}
& MN_{\rm gcd}(F({\bf g}),G({\bf g}),r) \\
&\le_{\rm exc} c_{m,n,d}  \sum_{i=1}^n N^{(L)}_{ g_i}(0,r)+ \left(\frac{m}{n+1}\binom{m+n}{n}-c_{m,n,d}-M'm\right)\sum_{i=1}^nN_{g_i}(0,r) \nonumber \\ 
  &+\binom {m+n-2d}{n}N_{\rm gcd}(\{{\bf g}^{\bf i}\}_{{\bf i}\in I},r)+ \left(M'mn+\epsilon m +\frac {M\epsilon}2\right)T_{\bf g}(r)+{\rm o}(T_{\bf g}(r)), \nonumber
\end{align} 
where $c_{m,n,d}=2\binom {m+n-d}{n+1}- \binom {m+n-2d}{n+1}$,
$M=2\binom {m+n-d}{n}- \binom {m+n-2d}{n}$, and $M'$ is an integer of order ${\rm O}(m^{n-2})$, where $\le_{\rm exc}$ means the inequality holds for all $r\in (0, \infty)$ except for a set $E$ of finite measure.\end{theorem}
  
\begin{proof}[Proof of Theorem \ref{Mgcdunit}]
Without loss of generality, we assume that $\deg F=\deg G$.
We first prove when $n\ge 2$. Let $\epsilon>0$.  To establish \eqref{gcd0} or  \eqref{gdegenerate}, we can assume that $\epsilon$ is sufficiently small.    We can  choose   a  real $C_1\ge 1$  independent of $\epsilon$ and ${\bf g}$ such that $  m= C_1\epsilon^{-1}\ge 2d$ and
\begin{align}\label{findm} 
		\frac{M'mn}{M}\le\frac{\epsilon}{4},\quad\text{and } \quad \frac1M\left(\frac{m}{n+1}\binom{m+n}{n}-c_{m,n,d}-M'm\right) \le\frac{\epsilon}{4(n+1)}.
	\end{align}

  We may assume that each $g_i$ is not identically zero; otherwise, \eqref{gdegenerate} holds trivially. 
	Suppose that  the set $\{g_0^{i_0}\cdots g_n^{i_n}: i_0+\cdots+i_n=m\}$ is linearly independent over $K_{\bf g}$.
We aim at concluding \eqref{gcd0} under assumption (a) or (b).
Suppose (a) holds, i.e.  $N_{g_i}(0,r)={\rm o}(T_{\bf g}(r))$ for $0\le i\le n$. Then (\ref{fundmenatlgcd}) implies that 
 		\begin{align}\label{gcdest}
		N_{\rm gcd}&(F({\bf g}),G({\bf g}),r) 
		\le_{\rm exc}  
		\left(\frac{M'mn}{M}+\epsilon \frac mM +\frac {\epsilon}2   \right)T_{\bf g}(r) 
		+{\rm o}(T_{\bf g}(r))< \epsilon T_{\bf g}(r).
	\end{align}
If (b) holds, then
$$
N^{(L)}_{ g_i}(0,r)\le LN^{(1)}_{ g_i}(0,r)={\rm o}(T_{\bf g}(r))
$$ for $0\le i\le n$.
The assumption that one of   $[G=0]$ or $[F=0]$  is in weakly general position with the $n+1$ coordinate hyperplanes in $\mathbb P^n$ implies that the set $\{(d,0,\hdots,0),\hdots, (0,\hdots,0,d)\}$ is a subset of $I$.  Since $g_0,\hdots,g_n$ are entire function with no common zero, we have 
$$
N_{\rm gcd}(\{{\bf g}^{\bf i}\}_{{\bf i}\in I},r)=0	
$$
when (b) holds.	
Then by \eqref{fundmenatlgcd}, \eqref{findm}  and  that	$N_{g_i}(0,r)\le T_{\bf g}(r)$, we obtain \eqref{gcd0}.
 	
	Finally, if  the set $\{g_0^{i_0}\cdots g_n^{i_n}: i_0+\cdots+i_n=m\}$ is dependent over $K_{\bf g}$, then  we may apply Lemma \ref{Mborel1}	to 	derive that there exists a non-trivial $n$-tuple of integers $(j_1,\dots,j_n)$ with $|j_1|+\cdots+|j_n|\le 2m$   such that 
	\begin{equation*}
		T_{(\frac{g_1}{g_0})^{j_1}\cdots (\frac{g_n}{g_0})^{j_n}}(r)= {\rm o}(T_{\bf g}(r)).
	\end{equation*}
 \end{proof}

\section{Proof of Theorem~\ref{main_thm_1GExc}}\label{generalZ123}
\subsection{Some  Lemmas}\label{mainlemmas} 
We recall some lemmas from \cite{GNSW}.
\begin{lemma}\label{main_lemma}
Let $n\ge 2$ and let $(m_1,\hdots,m_n)$ be  a non-zero vector in $\mathbb{Z}^n$
with $\gcd(m_1,\ldots,m_n)=1$. Then 
there exist $\bfv_i=(v_{i,1},\ldots,v_{i,n})\in \mathbb{Z}^{n}$ for $1\leq i\leq n-1$ such that
$$\vert v_{i,j}\vert \leq \max\{\vert m_j\vert,1\}\quad \text{for  $1\leq j\leq n$}$$
and  
$(m_1,\hdots,m_n)$ together with the $\bfv_i$'s form a basis of  $\mathbb Z^n$.
\end{lemma}
 
Let $k$ be a field and let $q$ and $r$ be positive integers. We write $\bft :=(t_1,\ldots,t_q)$ and $\bfx: =(x_1,\ldots,x_r) $.  For $\bfi=(i_1,\ldots,i_r)\in \mathbb{Z}^r$, denote $\bfx^{\bfi}=x_1^{i_1}\cdots x_r^{i_r}$ and $\bft^{\bfi}=t_1^{i_1}\cdots t_r^{i_r}$.  For $\Sigma\subseteq k[\bft]$, let $\mathcal{Z}(\Sigma)=\{\lambda\in k^q:\ f(\lambda)=0\ \text{for every $f\in \Sigma$.}\}$.   
\begin{lemma}\label{special} 
Assume that $k$ is infinite.
 Let $f(\bft,\bfx)\in k[\bft,\bfx]$ be a polynomial with no monomial factor and no repeated irreducible factor in $ k[\bft,\bfx]$.  Then there exists an effectively computable non-empty finite set
 $\Sigma\subset k[\bft]\setminus\{0\}$
 such that
 for every $\lambda\in k^q\setminus\mathcal{Z}(\Sigma)$, 
 the polynomial $f(\lambda,\bfx)$ has no monomial or repeated irreducible factor.
 Moreover, the cardinality of $\Sigma$ and the degree of each polynomial in $\Sigma$ can be bounded effectively in terms of $q$, $r$, and the degree of $f$.  
 Furthermore, if  $f(\bft,\bfx)\in k_0[\bft,\bfx]$ for $k_0$ being a subfield of $k$, then  $\Sigma$ is defined over $k_0$.
\end{lemma}

\subsection{Preliminary Theorem}
Let $\mathbf g=(g_0,\hdots,g_n)$, where $g_i\not\equiv 0$, $0\le i\le n$, are  entire functions without common zeros.
Let $u_i=g_i/g_0$, for $1\le i\le n$.
We observe that 
\begin{align}\label{htequiv}
 \max_{1\le j\le n}\{T_{u_j}(r)\}\le T_{\mathbf{g}}(r)\le n \max_{1\le j\le n}\{T_{u_j}(r)\},
\end{align}
and
\begin{align}\label{zeroquiv}
N_{u_i} (0,r)+N_{u_i}( \infty,r)\le N_{g_i}(0,r)+N_{g_0}(0,r)
\end{align}
for each $1\le i\le n$.

Recall that 
$$
  K_{\bf g}:=\{a: a \text{ is a meromorphic function on ${\Bbb C}$ with }  T_{a}(r)\le {\rm o} (T_{\mathbf g}(r))\},
$$ 
 which is the field of meromorphic functions of  slow growth with respect to  ${\bf g}$.     
We note that $a'\in  K_{\bf g}$ if $a\in   K_{\bf g}$. 
Furthermore, $u_i'/u_i\in K_{\bf g} $  if $N^{(1)}_{u_i} (0,r)+N^{(1)}_{u_i}( \infty,r)\le {\rm o} (T_{u_j}(r))$.

Let  $\mathbf{x}:=(x_{1},\ldots,x_{n})$ and $\mathbf u=(u_1,\hdots,u_n)$ .  For $\mathbf{i}=(i_{1},\ldots,i_{n})\in\mathbb{Z}^{n}$, we let $\mathbf{x^{i}}:=x_{1}^{i_{1}}\cdots x_{n}^{i_{n}}$ and  $\mathbf{u^{i}}:=u_{1}^{i_1}\cdots u_{n}^{i_{n}}$. For a non-constant polynomial $F(\mathbf{x})=\sum_{\mathbf{i}}a_{\mathbf{i}}\mathbf{x}^{\mathbf{i}}\in    K_{\bf g}[\mathbf{x}]:=    K_{\bf g}[x_1,\dots,x_n]$,
we define 
\begin{align}\label{Duexpression}
	D_{\mathbf{u}}(F)(\mathbf{x}):=\sum_{\mathbf{i}}\frac{(a_{\mathbf{i}}\mathbf{u}^{\mathbf{i}})'}{\mathbf{u}^{\mathbf{i}}}\mathbf{x}^{\mathbf{i}}
	=\sum_{\mathbf{i}}\left(a_{\mathbf{i}}'+a_{\mathbf{i}}  \cdot\sum_{j=1}^{n}i_j\frac{u_{j}'}{u_{j}}\right)\mathbf{x}^{\mathbf{i}}\in  K_{\mathbf g}[\mathbf{x}].
\end{align} 
A direct computation shows that 
\begin{align}\label{fuvalue}
	F(\mathbf{u})'=D_{\mathbf{u}}(F)(\mathbf{u}),
\end{align}
and that  the following product rule 
\begin{align}\label{productrule}
	D_{\mathbf{u}}(FG)=D_{\mathbf{u}}(F)G+FD_{\mathbf{u}}(G)
\end{align}
holds for  $F,G\in  K_{\mathbf g}[\mathbf{x}]$.

\begin{lemma}[ {\cite[Lemma 3.1]{GSW20}}]\label{coprimeD}
	Let $F$ be a nonconstant   polynomial in $ K_{\mathbf g}[\mathbf{x}]$  with no monomial factors and no repeated factors.  Assume that 
	$N_{u_i}^{(1)} (0,r)+N^{(1)} _{u_i}( \infty,r)={\rm o}\left(\max_{1\le j\le n}\{T_{u_j}(r)\}\right)$ for each $1\le i\le n$.
	Then   $F$   and $D_{\mathbf{u}}(F)$
	are coprime  in  $ K_{\mathbf g}[\mathbf{x}]$ unless there exists a non-trivial tuple of integers $(m_1,\hdots,m_n)$ with $\sum_{i=1}^n|m_i|\le 2\deg F$  such that  
	$ T_{u_1^{m_1} \cdots u_n^{m_n}}(r)={\rm o}\left(\max_{1\le j\le n}\{T_{u_j}(r)\}\right).$
\end{lemma}

 We now state a preliminary theorem in affine form.
 
\begin{theorem}\label{prethm}
Let $K$ be a subfield of the field of meromorphic functions.
Let $G$ be a nonconstant  polynomial in $K[x_1,\hdots,x_n]$ with no monomial  factors and no repeated factors.  For any   $\epsilon >0$, there exists a positive integer    $m$  such that
	for any $n$-tuple of meromorphic functions $\mathbf u=(u_1,\hdots,u_n)$ satisfying $ K\subset K_{\mathbf g}$, where $\mathbf g=[1:u_1:\cdots :u_n]$,  we  have  either  
\begin{align}\label{exception1}
		 T_{{u_1}^{m_1}\cdots  {u_n}^{m_n}}(r)={\rm o}\left(\max_{1\le j\le n}\{T_{u_j}(r)\}\right)
\end{align}
	for a non-trivial $n$-tuple $(m_1,\hdots,m_n)$ of integers with $\sum_{i=0}^n |m_i|\le  m$, or
 \begin{align}\label{multizerobddn}
 	N_{G(\mathbf{u})}(0,r)-N^{(1)}_{G(\mathbf{u})}(0,r)\le_{\rm exc} \epsilon \max_{1\le j\le n}\{T_{u_j}(r)\},
 \end{align}
if additionally one of the following holds:
\begin{enumerate}
\item[{\rm(a)}]
$N_{u_i} (0,r)+N_{u_i}( \infty,r)={\rm o}(\max_{1\le j\le n}\{T_{u_j}(r)\})$ for each $1\le i\le n$, or
\item[{\rm(b)}]$N_{u_i}^{(1)} (0,r)+N^{(1)} _{u_i}( \infty,r)={\rm o}(\max_{1\le j\le n}\{T_{u_j}(r)\})$ for each $1\le i\le n$,  and that  
$[G=0]$  and the $n+1$ coordinate hyperplanes are in weakly general position in $\mathbb P^n$.
\end{enumerate}
  \end{theorem}

\begin{proof}
Let $z_0\in\mathbb C$.  If $v_{z_0}( G(\mathbf{u}))\ge 2$, then it follows from \eqref{fuvalue} that $v_{z_0}( D_{\mathbf{u}}(G)(\mathbf{u}))= v_{z_0}( G(\mathbf{u}))-1$.
Hence,
$$
\min\{v_{z_0}^+(G(\mathbf{u})),v_{z_0}^+(D_{\mathbf{u}}(G)(\mathbf{u}))\}\ge v_{z_0}^+( G(\mathbf{u}))-\min\{1,v_{z_0}^+( G(\mathbf{u}))\}.
$$
Consequently,
\begin{align}\label{truncate}
N_{\gcd}(G(\mathbf{u}), D_{\mathbf{u}}(G)(\mathbf{u}),r)\ge N_{G(\mathbf{u})}(0,r)-N^{(1)}_{G(\mathbf{u})}(0,r).
\end{align}
By Lemma \ref{coprimeD}, $G$ and $D_{\mathbf{u}}(G)$ are either coprime or \eqref{exception1} holds for $m=2\deg G$.
Therefore, we assume that $G$ and $D_{\mathbf{u}}(G)$ are coprime.  By Theorem \ref{Mgcdunit}, we find a positive integer $m$ depending only on  $\epsilon$ such that either  \eqref{exception1} holds or 
\begin{align}\label{gcdUPB}
N_{\gcd}(G(\mathbf{u}), D_{\mathbf{u}}(G)(\mathbf{u}),r)\le_{\rm exc} \epsilon \max_{1\le j\le n}\{T_{u_j}(r)\}.
\end{align}
Together with \eqref{truncate}, we obtain \eqref{multizerobddn}.
\end{proof}

\subsection{Further refinement}
We will prove the following theorem by finding an exceptional set  in Theorem \ref{prethm}.
\begin{theorem}\label{keythm}
Let $K$ be a subfield of the field of meromorphic functions.
Let $G$ be a nonconstant  polynomial in $K[x_1,\hdots,x_n]$ with no monomial  factors and no repeated factors.  For any   $\epsilon >0$, there exists a nonconstant polynomial $H$ in $K[x_1,\hdots,x_n]$ such that for any $n$-tuple of meromorphic functions $\mathbf u=(u_1,\hdots,u_n)$ satisfying \begin{align}\label{uismall}
N_{u_i} (0,r)+N_{u_i}( \infty,r)={\rm o}\left(\max_{1\le j\le n}\{T_{u_j}(r)\}\right)\quad\text{for each $1\le i\le n$,}
\end{align} 
	and  $ K\subset K_{\mathbf g}$, where $\mathbf g=[1:u_1:\cdots :u_n]$,	we have either $H(\mathbf u)\equiv 0$ or
 \begin{align}\label{multizerobdd1}
 	N_{G(\mathbf{u})}(0,r)-N^{(1)}_{G(\mathbf{u})}(0,r)\le_{\rm exc} \epsilon \max_{1\le j\le n}\{T_{u_j}(r)\}.
 \end{align}
 Moreover, $H$ can be determined effectively and the degree of $H$ can be bounded effectively in terms of $\epsilon$, $n$ and the degree of $G$. 
  \end{theorem}
  \begin{proof} The proof of \cite[Theorem 4]{GNSW} can be adapted to suit the current situation. We will closely adhere to their arguments and notation. We first fix some notations: (i) For a matrix $A=(a_{ij})$ with complex-valued entries, let 
$$\Vert A\Vert_{\infty}=\max_{i}\sum_{j}\vert a_{ij}\vert$$
be the maximum of the absolute row sums;  (ii) We say that a non-trivial meromorphic function $\beta$ {\it has small zeros and poles w.r.t $\mathbf g$}  if 
$N_{\beta} (0,r)+N_{\beta}( \infty,r)= {\rm o}(T_{\mathbf g}(r))$.

Let $G\in K[x_1,\hdots,x_n]\setminus K$ with no monomial  factors and no repeated factors.  Let $\epsilon>0$.
In the following we consider a $n$-tuple of meromorphic functions $\mathbf u=(u_1,\hdots,u_n)$ satisfying
\begin{align*}
N_{u_i} (0,r)+N_{u_i}( \infty,r)={\rm o}\left(\max_{1\le j\le n}\{T_{u_j}(r)\}\right)={\rm o}\left(T_{\mathbf g}(r)\right)
\end{align*}
 for each $1\le i\le n$,
	and  $ K\subset K_{\mathbf g}$, where $\mathbf g=[1:u_1:\cdots :u_n]$.
We note that $\lambda\in K_{\mathbf g}$ if and only if $T_{\lambda}(r)={\rm o}\left(\max_{1\le j\le n}\{T_{u_j}(r)\}\right)$ 
by \eqref{htequiv}.

When $n=1$, the theorem is a direct consequence of Theorem \ref{prethm} since $u_1$ is constant if \eqref{exception1} holds.

From this point, we let $n\ge 2$.
We will effectively construct   a nonconstant polynomial $H$ in $K[x_1,\hdots,x_n]$ 
such that \eqref{multizerobdd1} holds if $H(u_1,\ldots,u_n)\not\equiv 0$.  

The arguments are carried out inductively in several steps. In the following, the $c_{i,j}$'s and $M_i$'s denote positive real numbers depending only on $\epsilon$, $n$, $\deg(G)$, and the previously defined $c_{i',j'}$ and $M_{i'}$. 

\textbf{Step 1}: We apply Theorem~\ref{prethm}. Note that the condition (a) in Theorem~\ref{prethm} holds under our assumption, so if (\ref{multizerobddn}) holds then we are done. Otherwise, 
 there exists an $n$-tuple of integers $(m_1,\ldots,m_n)\neq (0,\ldots,0)$ with $\sum\vert m_i\vert\leq M_1$ such that
\begin{equation}\label{eq:c_1,3 and c_1,4}
\lambda_1:=  u_1^{m_1}\cdots u_n^{m_n} \in  K_{\mathbf{g}}.
\end{equation}

We may assume $\gcd(m_1,\ldots,m_n)=1$. By Lemma \ref{main_lemma}, $(m_1,\hdots,m_n)$ extends to a basis 
$(m_1,\hdots,m_n)$,  $(a_{21},\hdots,a_{2n}),\hdots, (a_{n1},\hdots,a_{nn})$
of  $\mathbb Z^n$ such that 
\begin{align}\label{basisbound_step1}
|a_{i1}|+\cdots+|a_{in}|\le M_1+n\quad\text{ for } 2\le i\le n.
\end{align} 

Consider the change of variables
\begin{align}\label{transform_var_step1}
\Lambda_1:=x_1^{m_1}\cdots x_n^{m_n}, \quad\text{and } \quad
X_{1,i}:=x_1^{a_{i1}}\cdots x_n^{a_{in}}  \quad\text{for } 2\le i\le n
\end{align}
and put
\begin{align}\label{transform_unit_step1}
\beta_{1,i}=u_1^{a_{i1}}\cdots u_n^{a_{in}}  \quad\text{for } 2\le i\le n.
\end{align}

Let $A_1$ denote the $n\times n$ matrix whose rows are the above basis of $\mathbb{Z}^n$. Then we formally express the above identities as
\begin{equation}\label{eq:formal_A1}
(\Lambda_1,X_{1,2},\ldots,X_{1,n})=(x_1,\ldots,x_n)^{A_1}\quad \text{and} \quad (\lambda_1,\beta_{1,2},\ldots,\beta_{1,n})=(u_1,\ldots,u_n)^{A_1}.
\end{equation}
Let $B_1=A_1^{-1}$. The entries of $B_1$ can be bounded from above in terms of $M_1$ and $n$. We have
\begin{equation}\label{eq:formal_B1}
(x_1,\ldots,x_n)=(\Lambda_1,X_{1,2},\ldots,X_{1,n})^{B_1}\quad \text{and} \quad (u_1,\ldots,u_n)=(\lambda_1,\beta_{1,2},\ldots,\beta_{1,n})^{B_1}.
\end{equation}

Let $G_1(\Lambda_1,X_{1,2},\ldots,X_{1,n})\in K[\Lambda_1,X_{1,2},\ldots,X_{1,n}]$ with no monomial factors
and 
\begin{equation}\label{eq:G_1}
G((\Lambda_1,X_{1,2},\ldots,X_{1,n})^{B_1})=\Lambda_1^{d_1}X_{1,2}^{d_2}\cdots X_{1,n}^{d_n} G_1(\Lambda_1,X_{1,2},\ldots,X_{1,n})
\end{equation}
for some integers $d_i$, $1\leq i\leq n$. Since the transformations in \eqref{eq:formal_A1} and \eqref{eq:formal_B1} are invertible of each other and $G$ has no repeated irreducible factors, we have that $G_1$ has no repeated irreducible factors either. The coefficients of $G_1$ are the same as the coefficients of $G$ and $\deg(G_1)$ can be bounded from above explicitly in terms of $M_1$, $n$, and $\deg(G)$. Consider $G_1(\lambda_1,X_{1,2},\ldots,X_{1,n})\in K(\lambda_1)[X_{1,2},\ldots,X_{1,n}]$, by using \eqref{eq:c_1,3 and c_1,4} we have
\begin{equation}\label{eq:tilde h after specialization 1}
K(\lambda_1)\subset K_{\mathbf{g}}.
\end{equation}

For the particular change of variables in \eqref{eq:formal_A1}, \eqref{eq:formal_B1}, and \eqref{eq:G_1} (that depends on the matrix $A_1$),
we apply the Lemma \ref{special}  with $k$ being the field of meromorphic functions $\mathcal M$ and $k_0=K$ and \eqref{transform_var_step1} to find a nonconstant polynomial $H_1'\in K[x_1,\hdots,x_n]$ 
such that   $G_1(\lambda_1,X_{1,2},\ldots,X_{1,n})$
has neither monomial nor repeated irreducible factors  if $H_1'(u_1,\ldots,u_n)\not\equiv  0$. We now take $H_1$ to be the product of all
such $H_1'$ where $A_1$ ranges over the finitely many elements of $\GL_n(\mathbb{Z})$ with $\Vert A_1\Vert_{\infty}\leq M_1+n$. From Lemma \ref{special}, we know that $\deg H_1$ depends only on $\epsilon$, $n$ and $\deg G$.

Since the $u_i$'s, $\lambda_1$, and $\beta_{1,j}$'s have small zero and pole with respect to $\mathbf{g}$, we have:
\begin{align}\label{eq:N_S-N_S^1 step1}
\begin{split}
  N_{G(\mathbf u)}(0,r)-N_{G(\mathbf u)}^{(1)}(0,r) 
=  N_{G_1(\lambda_1,\beta_{1,2},\ldots,\beta_{1,n})}(0,r)-N_{G_1(\lambda_1,\beta_{1,2},\ldots,\beta_{1,n})}^{(1)}(0,r)+{\rm o}( T_{\mathbf{g}}(r)) 
\end{split}
\end{align}
by  \eqref{eq:formal_A1}  and \eqref{eq:G_1}.
From  \eqref{eq:formal_A1}, \eqref{eq:formal_B1} and \eqref{eq:c_1,3 and c_1,4}, we have:
\begin{equation}\label{eq:c_1,78910}
 \max_{1\leq i\leq n}\{T_{u_i}(r)\} ={\rm O}\left(\max\{T_{\lambda_1}(r), T_{\beta_{1,2}}(r)),  \ldots,T_{\beta_{1,n}}(r) \}\right)={\rm O}\left(\max_{2\leq i\leq n}\{T_{\beta_{1,i}}(r)\}\right). 
\end{equation} 
In conclusion,  at the end of this step we have 
\begin{align}\label{eq:end step1 1}
	\max_{2\leq i\leq n}\{T_{\beta_{1,i}}(r)\}={\rm O}\left( \max_{1\leq i\leq n}\{T_{u_i}(r)\}\right).
\end{align}
Furthermore, it remains to consider the case when
\begin{align}\label{eq:end step1 2}
 N_{G_1(\lambda_1,\beta_{1,2},\ldots,\beta_{1,n})}(0,r)-N_{G_1(\lambda_1,\beta_{1,2},\ldots,\beta_{1,n})}^{(1)}(0,r)&<_{\rm exc}\epsilon   \max_{1\leq i\leq n}\{T_{u_i}(r)\}
\end{align}
fails to hold under the assumption that $H_1(u_1,\ldots,u_n)\not\equiv 0$.

There are $n-1$ many steps in total. Hence if $n\geq 3$, we proceed with the following $n-2$ many more steps.

\textbf{Step 2}: We include this step in order to illustrate the transition from Step $s-1$ to Step $s$
below. Since the various estimates and constructions are similar to those in Step 1, we skip some of the details. Suppose $H_1(u_1,\ldots,u_n)\not\equiv 0$ so that $G_1(\lambda_1,X_{1,2},\ldots,X_{1,n})$ has
 neither monomial nor repeated factors. 
 
We now apply
 Theorem~\ref{prethm}, assuming (\ref{eq:end step1 2}) fails to hold, 
 for $G_1(\lambda_1,X_{1,2},\ldots,X_{1,n})$ and $(\beta_{1,2},\ldots,\beta_{1,n})$
 and use \eqref{eq:tilde h after specialization 1}, \eqref{eq:end step1 1},  to get an $(n-1)$-tuple
 $(m_2',\ldots,m_n')\neq (0,\ldots,0)$ with $\sum \vert m_i'\vert\leq M_2$ such that
  \begin{equation}\label{eq:c_2,3 and c_2,4}
\lambda_2:=\beta_{1,2}^{m_2'}\cdots \beta_{1,n}^{m_n'}\in K_{\mathbf{g}}.
\end{equation}
 We may assume $\gcd(m_2',\ldots,m_n')=1$. By Lemma \ref{main_lemma}, $(m_2',\hdots,m_n')$ extends to a basis of $\mathbb{Z}^{n-1}$ in which each vector has $\ell_1$-norm at most $M_2+n$.

Let $A_2'$ be the $(n-1)\times (n-1)$ matrix whose rows are the above basis of $\mathbb{Z}^{n-1}$. We make
the transformation:
$$(\Lambda_2,X_{2,3},\ldots,X_{2,n})=(X_{1,2},\ldots,X_{1,n})^{A_2'} \quad \text{and} \quad
(\lambda_2,\beta_{2,3},\ldots,\beta_{2,n})=(\beta_{1,2},\ldots,\beta_{1,n})^{A_2'}.$$
Let $A_2=(1)\oplus A_2'$ be the $n\times n$ block diagonal matrix with the $(1,1)$-entry $1$ and the matrix $A_2'$ in the remaining $(n-1)\times (n-1)$ block. We have
$$(\Lambda_1,\Lambda_2,\ldots,X_{2,n})=(\Lambda_1,X_{1,2},\ldots,X_{1,n})^{A_2} \quad \text{and} \quad
(\lambda_1,\lambda_2,\ldots,\beta_{2,n})=(\lambda_1,\beta_{1,2},\ldots,\beta_{1,n})^{A_2}.$$
Combining this with \eqref{eq:formal_A1}, we have:
\begin{align}\label{eq:formal_A2}
\begin{split}
(\Lambda_1,\Lambda_2,X_{2,3},\ldots,X_{2,n})=(x_1,\ldots,x_n)^{A_2A_1}\quad \text{and}\\
(\lambda_1,\lambda_2,\beta_{2,3},\ldots,\beta_{2,n})=(u_1,\ldots,u_n)^{A_2A_1}.
\end{split}
\end{align}
Let $B_2=(A_2A_1)^{-1}$. Let $G_2(\Lambda_1,\Lambda_2,X_{2,3},\ldots,X_{2,n})$ be the polynomial with no monomial factors such that
$$G_0((\Lambda_1,\Lambda_2,X_{2,3},\ldots,X_{2,n})^{B_2})=\Lambda_1^{d_1'}\Lambda_2^{d_2'}X_{2,3}^{d_3'}\cdots X_{2,n}^{d_n'}G_2(\Lambda_1,\Lambda_2,X_{2,3},\ldots,X_{2,n})$$
for some $d_1',\ldots,d_n'\in\mathbb{Z}$. We have that $\deg(G_2)$ can be bounded from above explicitly in terms of $M_2$, $M_1$, $n$, and $\deg(G)$. As before, we regard $G_2(\lambda_1,\lambda_2,X_{2,3},\ldots,X_{2,n})$ as a polynomial in $X_{2,3},\ldots,X_{2,n}$  with coefficients in $K_{\mathbf{g}}$ using \eqref{eq:c_1,3 and c_1,4} and \eqref{eq:c_2,3 and c_2,4}.

 For a particular $A_1$ and $A_2$, we apply Lemma~\ref{special}  with $k=\mathcal M$ and $k_0=K$ and use \eqref{eq:c_1,3 and c_1,4} and \eqref{eq:c_2,3 and c_2,4} to get  a nonconstant polynomial $H_2'$ in $K[x_1,\hdots,x_n]$ such that $G_2(\lambda_1,\lambda_2,X_{2,3},\ldots,X_{2,n})$ has neither monomial nor repeated factors. 
We now take $H_2$ to be the product of all such $H_2'$ where $A_1$ and $A_2$
 range over the finitely many unimodular matrices with
 $\Vert A_1\Vert_{\infty}\leq M_1+n$ and $\Vert A_2\Vert_{\infty}\leq M_2+n$.   
 By using similar estimates, at the end of this step, we have
 \begin{align}\label{eq:end step2 1}
\max_{3\leq i\leq n}\{T_{\beta_{1,i}}(r)\}={\rm O}\left( \max_{1\leq i\leq n}\{T_{u_i}(r)\} \right)
\end{align}
and
\begin{align}\label{eq:end step2 3}
 N_{G_2(\lambda_1,\lambda_2,\beta_{2,3},\ldots,\beta_{2,n})}(0,r)-N_{G_2(\lambda_1,\lambda_2,\beta_{2,3},\ldots,\beta_{2,n})}^{(1)}(0,r))<_{exc}\epsilon  \max_{1\leq i\leq n}\{T_{u_i}(r)\}
\end{align}
fails to hold.

Let $2\leq s\leq n-1$ and suppose that we have completed Step $s-1$. This includes the construction
of $H_{s-1}\in K[x_1,\hdots,x_n]$ with degree depends on $\epsilon$, $n$ and $\deg G$ only. We then complete Step $s$ in the same manner Step $2$ is carried out after Step 1. The last one is Step $n-1$ resulting in $H_{n-1}\in K[x_1,\hdots,x_n]$. We now define
$H=H_1 \cdots H_{n-1}$.  Then $\deg H$  depends only on $\epsilon$, $n$ and $\deg G$ since each $H_i$ does so.
Suppose $H(u_1,\ldots,u_n)\not\equiv 0$.   
Assume we go through all the above $(n-1)$ steps
to get the polynomial
$$P(X_{n-1,n}):=G_{n-1}(\lambda_1,\ldots,\lambda_{n-1},X_{n-1,n})\in K_{\mathbf g}[X_{n-1,n}]$$ 
such that its degree can be bounded explicitly in terms of $M_{n-1},\ldots,M_1$, $n$, and $\deg(G)$. At the end of Step $n-1$, we have that $\beta_{n-1,n}$  has small zero and pole with respect to $\mathbf{g}$, so it satisfes
 \begin{align}\label{eq:end step n-1 1}
 T_{\beta_{n-1,n}}(r) ={\rm O}\left( \max_{1\leq i\leq n}\{T_{u_i}(r)\}\right). 
\end{align}
If 
\begin{align}\label{eq:end step n-1 2}
N_{P(\beta_{n-1,n})}(0,r)-N^{(1)}_{P(\beta_{n-1,n})}(0,r)
<_{exc} \epsilon   \max_{1\leq i\leq n}\{T_{u_i}(r)\} \end{align}
holds, then we are done. Otherwise, 
since $H_{n-1}(u_1,\ldots,u_n)\not\equiv 0$, the polynomial $P(X_{n-1,n})$
has neither monomial nor repeated irreducible factors,  according to Theorem~\ref{prethm},   there exists a non-zero integer $m$ such that, by using (\ref{eq:end step n-1 1}), 
\begin{align}\label{eq:cn3 and cn4}
	T_{\beta_{n-1,n} ^m}(r) ={\rm o}( T_{\beta_{n-1,n}}(r)), 
\end{align}
which is not possible since $\beta_{n-1,n}$ is not constant.  This completes the proof.
  \end{proof} 
  
\subsection{Proof of Theorem~\ref{main_thm_1GExc}} 
  
\begin{proof}[Proof of Theorem~\ref{main_thm_1GExc}]
  
  Let $F\in K[x_0, \ldots, x_n]$. 
Consider a holomorphic curve $\mathbf{g}=(g_0,\hdots,g_n)$, where $g_0, \dots, g_n$ are entire functions with no common zeros,  such that $K\subset K_{\mathbf{g}}$ and $N_{\mathbf{g}}(H_i,r)={\rm o}( T_{\mathbf{g}}(r))$ for $0\le i\le n$.
Let $u_i=\frac{g_i}{g_0}$ for $0\le i\le n$, $\mathbf{u}=(u_1,\ldots,u_n)$, and $G:=F(1,x_1,\hdots,x_n)\in K[x_1,\hdots,x_n]$.
Then 
\begin{align}\label{zeroquiv1}
N_{u_i} (0,r)+N_{u_i}( \infty,r)\le N_{g_i}(0,r)+N_{g_0}(0,r)=N_{\mathbf{g}}(H_i,r)+N_{\mathbf{g}}(H_0,r)={\rm o}( T_{\mathbf{g}}(r))
\end{align}
for each $1\le i\le n$, and, by \eqref{htequiv}, 
\begin{equation}\label{CharFun}
 \max_{1\leq i\leq n}\{T_{u_i}(r)\} ={\rm O}( T_{\mathbf{g}}(r)).
 \end{equation} 
 Since $F(\mathbf{g})=F(g_0,\hdots,g_n)=g_0^dG(\mathbf{u})$, we have 
 \begin{equation}\label{affineCountingFun}
N_{F(\mathbf{g})}(0,r)=N_{G(\mathbf{u})}(0,r)+{\rm o}( T_{\mathbf{g}}(r))
 \quad \text{and} \quad 
N^{(1)}_{F(\mathbf{g})}(0,r)=N^{(1)}_{G(\mathbf{u})}(0,r)+{\rm o}( T_{\mathbf{g}}(r)).
 \end{equation} 
Consequently, we may apply Theorem \ref{keythm} for any given positive real $\epsilon$ to find a non-trivial polynomial $Q\in K[x_1, \ldots, x_n]$ such that \eqref{multizerobdd} holds, that is 
 \begin{equation}\label{multizerobddF}  
N_{F(\mathbf{g})}(0,r)- N^{(1)}_{F(\mathbf{g})}(0,r)\le \epsilon T_{\mathbf{g}}(r), 
 \end{equation} 
 when $Q(\mathbf{u})\not\equiv 0$.
In addition,  the polynomial  $Q$ can be determined effectively and the degree of $Q$ can be bounded effectively in terms of $\epsilon$, $n$ and the degree of $F$.  At this step, we take $Z$ to be the zero locus of the homogeneous polynomial $x_0^{\deg Q}\cdot Q(\frac{x_1}{x_0},\hdots,\frac{x_n}{x_0})\in K[x_0,\hdots,x_n]$.
 
Let $F=\sum_{\mathbf i\in I_F}\alpha_{\mathbf i}{\mathbf x}^{\mathbf i}\in K[ x_0,\hdots, x_n]$,   and let $W$ be the Zariski closed subset that is the union of hypersurfaces of $\mathbb P^n$ of the form  $\sum_{\mathbf i\in J}\alpha_{\mathbf i}{\mathbf x}^{\mathbf i}=0$, where $J$ is a non-empty subset of $I_F$.
The Zariski closed set $Z\cup W$ satisfies {\rm(Z1)} and {\rm(Z2)} since both $Z$ and $W$ do so.
We now prove \eqref{truncate1_2} holds (after possibly enlarging $Z$) by further assuming that the hypersurface $[F=0]$ and the coordinate hyperplanes  in $\mathbb P^n$ are in weakly general position.
 Therefore, may write  
\begin{align}\label{expressF}
F(\mathbf g) = \sum_{0\le i\le n}\alpha_{\mathbf i_i}g_i^d+\sum_{\mathbf i\in I_G\setminus I }\alpha_{\mathbf i}{\mathbf g}^{\mathbf i} 
\end{align}
where $\alpha_{\mathbf i_i}\ne 0$ for $0\le i\le n$ and $I=\{ \mathbf i_0:=(d,0,\hdots,0),\hdots, \mathbf i_n:=(0,\hdots,0,d)\}$.

  For  $ {\mathbf g}$ with ${\mathbf g}({\Bbb C})$ is not contained in $Z\cup W$, we may use Theorem~\ref{trunborel} to show that  
\begin{equation*}
		d T_{\mathbf g}(r)\le  N_{F(\mathbf g)}(0,r)+{\rm o}( T_{\mathbf{g}}(r)) 
	\end{equation*}  
 since $\alpha_{\mathbf i_i}\in K_{\mathbf g}$  and $N_{g_i}(0,r)={\rm o}( T_{\mathbf{g}}(r))$ for $0\le i\le n$.
Together with \eqref{multizerobddF}, we arrive at $ N_{F( {\mathbf g})}^{(1)}(0,r)\ge (d-\epsilon)T_{\mathbf{g}}(r)$. By letting $Z\cup W$ be the desired exceptional set $Z$, we finish the proof.
 \end{proof}

\section{Proof of Theorem  \ref{GG_conj}}
We will adapt the proof strategy employed in \cite[Theorem 1.2]{GSW22} to suit the current situation and subsequently apply Theorem \ref{main_thm_1GExc}.
\begin{proof}
Let $z_{0}\in\mathbb{C}$   such that  all the coefficients of  all $F_i$, $1\le i\le n+1$ are  holomorphic at $z_0$ and the zero locus of  $F_i$,   $1\le i\le n+1$, evaluating at $z_0$, denoted by $D_i(z_0)$,  intersect transversally.  We note these conditions imply that $z_0$ is not a common zero of  the coefficients of $F_i$, for each $1\le i\le n+1$.

 Since the zero locus of $F_{i}(z_{0})$, $1\le i\le n+1$, intersect transversally,
they are in general position; thus the set of polynomials  $F_{i}$,
$1\le i\le n+1$, is in  weakly general position. 
Then Proposition \ref{HilbertN} implies that the
only $(x_{0},\ldots,x_{n})\in\mathcal{M}^{n+1}$ with
$F_{i}(x_{0},\ldots,x_{n})\equiv 0$ for each $1\le i\le n+1$ is $(0,\ldots,0)$.
Thus the association \textcolor{black}{${\mathbf x}\mapsto[F_{1}^{a_{1}}({\mathbf x}):\hdots:F_{n+1}^{a_{n+1}}({\mathbf x})]$,
where $a_{i}\coloneqq{\rm lcm}(\deg F_{1},\hdots,\deg F_{n+1})/\deg F_{i}$,
defines a morphism} $\pi:\PP^{n}(\mathcal M)\to\PP^{n}(\mathcal M)$
over $K$.  Let 
\begin{align*}
G:=\det \left(\frac{\partial F_i }{\partial x_j}\right)_{1\le i\le n+1, 0\le j\le n}\in K[x_{0},\hdots,x_{n}].
 \end{align*} 
Denote by 
$\pi|_{z_{0}}=[F_{1}^{a_{1}}(z_{0}) :\dots:F_{n+1}^{a_{n+1}}(z_{0}) ]:\PP^{n}(\mathbb{C})\rightarrow\PP^{n}(\mathbb{C})$, which is a morphism since $F_1(z_0),\hdots,F_{n+1}(z_0)$ are in general position.
As proved in \cite[Theorem 1.2]{GSW22}, we have that $[G(z_0)=0]$ (the zero locus of $G(z_0)$), $D_1(z_0),\hdots,D_{n+1}(z_0)$ are  is in general position (in $\mathbb P^n(\CC)$).  Hence, there is a nonconstant irreducible factor $\tilde G$ of $G$  in $K[x_{0},\hdots,x_{n}]$ such that $\tilde G$,  $F_{1},\ldots,F_{n+1}$
is in weakly general position.   Denote by $Y$ the zero locus of $\tilde G$ in $ \PP^{n}(\overline K) $.  We note that  $Y$ is contained in the ramification divisor of $\pi$ since $G$ is a factor of the determinant of the Jacobian matrix  associated with the map $\pi$.  
Then there exists an irreducible homogeneous  polynomial $A\in K[y_0,\hdots,y_n]$ such that the vanishing order of $\pi ^*A$ along $Y$ is at least 2.  Then this construction gives
$\pi^*\circ  A=\tilde G^2H$ for some $H\in  K[x_{0},\hdots,x_{n}]$. 
Since  the divisors defined by $\tilde G(z_0)$,  $F_{1}(z_0),\ldots,F_{n+1}(z_0)$ are in general position, their images are also in general position.  Therefore,  $ A $ and $y_i$, $0\le i\le n$ are in weakly general position. 

Now let $\mathbf{f}=(f_{0},\hdots,f_{n}):\CC\to\PP^{n}$
be a holomorphic map, where $f_{0},\hdots,f_{n}$ are entire functions
without common zeros.  Assume that $K\subset K_{\mathbf{f}}$. 
Let  $\mathbf{g}\coloneqq\pi(\mathbf{f})=(F_{1}(\mathbf{f})^{a_{1}},\ldots,F_{n+1}(\mathbf{f})^{a_{n+1}})$
is a tuple of entire functions without zeros.  Then
\begin{align}\label{height}
	 T_{\mathbf{g}}(r)=d_1 T_{\mathbf{f}}(r)+{\rm o}(  T_{\mathbf{f}}(r)),
\end{align}
where $d_1=\deg F_1\cdot a_1$.
From the equality $ A(\mathbf{g})=(\pi^*\circ  A)(\mathbf{f})=\tilde G^2(\mathbf{f})H(\mathbf{f})$, it follows that for each $z\in\mathbb{C}$ with $v_{z}(\tilde G(\mathbf{f}))>0$,  
we have 
\begin{align}\label{poleH}
\max\{0,v_{z}(  A(\mathbf{g}))\}\ge 2v_{z}(\tilde G(\mathbf{f}))+\min\{0,v_{z}( H(\mathbf{f}))\}\ge v_{z}(\tilde G(\mathbf{f}))+1+\min\{0,v_{z}( H(\mathbf{f}))\}.
\end{align}
Since  $f_{0},\hdots,f_{n}$ are entire functions, the nonnegative number $-\min\{0,v_{z}( H(\mathbf{f}))\}$ is bounded by the number of poles of the coefficients of $H$ at $z$.   Since the coefficients of $H$ are in $K$ and that $N_{\beta}(\infty,r)\le T_{\beta}(r)+{\rm O}(1)={\rm o}(  T_{\mathbf{f}}(r))$ for any $\beta\in K$, it follows from \eqref{poleH}  that  
\begin{align}\label{countingG}
N_{\tilde G(\mathbf{f})}(0,r)\le N_{A(\mathbf{g})}(0,r)-N_{A(\mathbf{g})}^{(1)}(0,r)+{\rm o}(  T_{\mathbf{f}}(r)).
\end{align}
Assume furthermore that $N_{F_i(\mathbf{f})}(0,r)={\rm o}( T_{\mathbf{f}}(r))$ for $1\le i\le n+1$.  Then  $N_{\mathbf{g}}(H_i,r)={\rm o}( T_{\mathbf{f}}(r))$($={\rm o}( T_{\mathbf{g}}(r))$ by \eqref{height} ) for coordinate hyperplanes $H_i$, $0\le i\le n$, of $\mathbb P^n$.   We now apply Theorem \ref{main_thm_1GExc}  for  $ \epsilon =\frac{1}{4 d_1 }$.  Then we can  find a homogeneous polynomial $B_0\in  K[y_0,\ldots,y_n]$ such that for any nonconstant holomorphic map $\mathbf{f}=(f_0,\hdots,f_n):\CC\to \PP^n$  such that $K\subset K_{\mathbf{f}}$ and $N_{F_i(\mathbf{f})}(0,r)={\rm o}( T_{\mathbf{f}}(r))$ for $1\le i\le n+1$, with $B_0(\mathbf g)=B_0(\pi(\mathbf{f}))$ not identically zero, we have 
\begin{align}\label{multizerobdd}
	 N_{A(\mathbf{g})}(0,r)-N_{A(\mathbf{g})}^{(1)}(0,r)\le_{\rm exc} \epsilon T_{\mathbf{g}}(r),
\end{align}
and
\begin{align}\label{truncate1_2}
N_{A(\mathbf{g})}^{(1)}(0,r)\ge_{\rm exc}  (\deg  A- \epsilon)\cdot T_{\mathbf g}(r).
\end{align}
Combining  \eqref{countingG} and \eqref{multizerobdd}, we have 
\begin{align}\label{countingG2}
N_{\tilde G(\mathbf{f})}(0,r)  \le_{\rm exc} \epsilon  T_{\mathbf{g}}(r).
\end{align}
Since $[\tilde G=0]\le \pi^* ([ A=0])$ as divisors,
we can derive from the functorial property of Weil functions the following:
\begin{align}\label{prox}
m_{\mathbf f}([\tilde G=0],r) \le m_{ \mathbf{g} }([A=0],r)=\deg A \cdot T_{\mathbf{g}}(r)-N_{A(\mathbf{g})}(0,r)+{\rm o}( T_{\mathbf{g}}(r)).
\end{align}
Then by 	\eqref{truncate1_2},we have
\begin{align}\label{proxiG}
	  m_{\mathbf f}([\tilde G=0],r)\le \epsilon T_{\mathbf{g}}(r).
	\end{align}
Combining \eqref{countingG2}, \eqref{proxiG}  and \eqref{height}, we have 
\begin{align}\label{FG}
		T_{[\tilde G=0], \mathbf{f}}(r) \le_{\exc} 2\epsilon T_{\mathbf{g}}(r)=2\epsilon\cdot  d_1 T_{\mathbf{f}}(r)+{\rm o}( T_{\mathbf{f}}(r)), 
	\end{align}
On the other hand, the first main theorem implies that 
\begin{align}\label{FG}
		\deg \tilde G\cdot T_{\mathbf{f}}(r)=T_{[\tilde G=0], \mathbf{f}}(r)+{\rm o}( T_{\mathbf{f}}(r)).	
		\end{align}
Therefore, we have
\begin{align}\label{FG}
		T_{\mathbf{f}}(r) \le_{\exc}  2\epsilon\cdot  d_1 T_{\mathbf{f}}(r)+{\rm o}( T_{\mathbf{f}}(r)), 
	\end{align}
which is not possible since   $\epsilon=\frac{1}{4 d_1 }$.
This shows that  $B_0(\mathbf g)$ is identically zero.  Let $B:=\pi^*(B_0)=B_0(F_{1}^{a_{1}},\hdots,F_{n+1}^{a_{n+1}})\in K[x_0,\hdots,x_n]$, which is not identically zero since $\pi$ is a finite morphism.   
Then $B(\mathbf f)$ is identically zero as asserted.
  \end{proof}
 
\section{Proof of Theorem \ref{defect}}
The defect relation stated in Theorem \ref{defect} directly follows from Theorem \ref{GG_conj} by noticing that
$\sum_{i=1}^{n+1}\delta_{\mathbf{f}} (D_i)=n+1$ if and only if $N_{\mathbf{f}} (D_i, r)={\rm o}( T_{\mathbf{f}}(r))$ for each $i$.
To establish the truncated defect relation for $n=2$,  we relax the assumption to $N_{\mathbf{g}}^{(1)}(H_i,r)={\rm o}( T_{\mathbf{g}}(r))$ for $0\le i\le 2$. We state  a modified version of Theorem \ref{main_thm_1GExc} below
to demonstrate that Theorem \ref{GG_conj} remains valid under these conditions.

\begin{theorem}\label{main_thm_n2}
Let $K$ be a subfield of the field $\mathcal M$ of meromorphic functions.
Let $G$ be a non-constant  homogeneous polynomial   in $K[x_0,x_1,x_2]$ with no monomial factors and no repeated factors.  Let  $H_i=[x_{i-1}=0]$, $1\le i\le 3$,  be the  coordinate hyperplane divisors of   $ \mathbb P^2$.  Assume that the plane curve  $[G=0]$  and $H_i$, $1\le i\le 3$, are in weakly general position.  
Then for any   $\epsilon >0$, there exists a proper Zariski closed subset $Z$ of $\mathbb P^2$ defined over $K$ such that for any nonconstant holomorphic curve   $\mathbf{g}=(g_0,g_1,g_2):\CC\to\PP^2(\CC)$ such that   $N_{\mathbf{g}}^{(1)}(H_i,r)={\rm o}( T_{\mathbf{g}}(r))$ for $0\le i\le 2$ with $\mathbf g$ not contained in $Z$, we have   
\begin{align}\label{multizerobddn2}
	N_{G(\mathbf{g})}(0,r)-N^{(1)}_{G(\mathbf{g})}(0,r)\le_{\rm exc} \epsilon T_{\mathbf{g}}(r),
\end{align}
and
\begin{align}\label{truncate1_n2}
N^{(1)}_{G(\mathbf g)}(0,r)\ge_{\rm exc}  (\deg  G- \epsilon)\cdot T_{\mathbf g}(r).
\end{align}
Furthermore, the exceptional set $Z$ is a finite union of closed subsets given by homogenization equations  of the form $ x_1^{n_1}x_2^{n_2}=\lambda $, where     $\lambda\in K^*$ and  $(n_1,n_2)$ is a pair of integers with $\max\{|n_1|,|n_2|\} $ bounded from above by an effectively computable integer $m$.
\end{theorem}

\begin{proof}[Proof of Theorem \ref{defect}]
Since $0\le \delta_{\mathbf{f}}(D_i)\le 1$ for $1\le i\le n+1$, it is clear that 
$\sum_{i=1}^{n+1}\delta_{\mathbf{f}} (D_i)=n+1$ if and only if $\delta_{\mathbf{f}} (D_i)=1$ for each $i$.
On the other hand, $\delta_{\mathbf{f}}(D_i)=1$ if and only if $N_{\mathbf{f}}(D_i,r)={\rm o}( T_{\mathbf{f}}(r))$.
Therefore, we have either $\sum_{i=1}^{n+1}\delta_{\mathbf{f}}(D_i)<n+1$ or  there exists a homogeneous polynomial $B\in K[x_0,\hdots,x_n]$ as described  Theorem \ref{GG_conj} such that $B(\mathbf{f})$ is identically zero. 

When $n=2$, the conclusion of Theorem \ref{GG_conj} holds under a weaker assumption that $N^{(1)}_{\mathbf{f}}(D_i,r)={\rm o}( T_{\mathbf{f}}(r))$ for $i=0,1,2$ by replacing the use of Theorem  \ref{main_thm_1GExc} with Theorem \ref{main_thm_n2}.  Therefore, the above arguments show that $\sum_{i=1}^{3}\delta^{(1)}_{\mathbf{f}}(D_i)<3$  or 
 there exists a homogeneous polynomial $B\in K[x_0,x_1,x_2]$ as described  Theorem \ref{GG_conj} such that $B(\mathbf{f})$ is identically zero. 

\end{proof}

\begin{proof}[Proof of Theorem \ref{main_thm_n2}] Let  $\mathbf{g}=(g_0,g_1,g_2)$ with $N^{(1)}_{g_i} (0,r) ={\rm o}(T_{\mathbf g}(r))$, $0 \le i \le 2$, where $g_0, g_1, g_2$ 
have no common zeros. 
We  prove \eqref{multizerobddn2} first.  Note that under our assumption, the condition (b) in Theorem \ref{prethm} holds. 
Hence, by Theorem \ref{prethm},  we only need to consider the case that 
\begin{align}\label{exception01}
		 T_{(\frac{g_1}{g_0})^{n_1}  (\frac{g_2}{g_0})^{n_2}}(r)={\rm o}( T_{\mathbf g}(r) )
\end{align}
holds. 

We may assume that  $n_1$ and $n_2$ are coprime. Consequently, there exist integers $a$ and $b$ such that $n_1a+n_2b=1$. 
Consider the variables
	\begin{align}\label{LamG}
		\Lambda=X^{n_{1}}  Y^{n_{2}}\qquad\text{and}\qquad  T=X^bY^{-a}.
	\end{align} 
Then, we may express
	\begin{align}\label{XY}
		X=\Lambda^a T^{n_2}\qquad\text{and}\qquad  Y=\Lambda^bT^{-n_1}.
	\end{align} 
 Let $G_1(X,Y)=G(1,X,Y)$.
Define   $B(\Lambda,T)\in K[\Lambda,T]$ as the polynomial with no monomial factors and such that  \begin{align}\label{GB}
		G_1(X,Y)=G_1(\Lambda^a T^{n_2},\Lambda^b T^{-n_1})= T^{M_1}\Lambda^{M_2}B(\Lambda,T)
	\end{align}
	for some integers $M_1$ and $M_2$.

Let $u_1=g_1/g_0$, $u_2=g_2/g_0$,  and $\lambda:=u_{1}^{n_{1}}  u_{2}^{n_{2}}$. 	 Then we have
\begin{align}\label{Tlambda}
		T_{\lambda}(r)={\rm o}( T_{\mathbf g}(r) ). 
	\end{align}

To prove \eqref{multizerobddn2},  we will reduce the problem to one-variable polynomials $B(\lambda,T)$ for all possible $\lambda\in K$ that satisfy \eqref{Tlambda} but not \eqref{multizerobddn2}.   Our objective is to eliminate those $\lambda$ values with $B(\lambda, T)$ containing a factor of $T$ or having repeated factors, so that we can apply the GCD theorem after eliminating those $\lambda$.
Since $T$ is not a factor of  $B(\Lambda,T)$, it follows that  $B(\Lambda,0)\in K[\Lambda]$ is not identically zero.  Consequently, there exist at most finite	$ \gamma_1,\dots,\gamma_s\in   K$ such that $B(\gamma_i,0)=0$ for $1\le i\le s.$ Therefore, $T$ is not a factor of $B(\lambda, T)$ if $\lambda\ne \gamma_i$, $1\le i\le s$.  
Regarding  repeated factors,  let's express $B(\Lambda,T)=B_{\Lambda}(T)\in K[\Lambda][T]$. 
Since the transformation in \eqref{LamG} establishes to a bijection between the sets $\{X^{t_1}Y^{t_2}: t_1,t_2\in\mathbb Z\}$ and $\{\Lambda^{a_1}T^{a_2}: a_1,a_2\in\mathbb Z\}$,    it is evident  that $B(\Lambda,T)\in K[\Lambda,T]$  is square free, given that $G$ is square-free. Consequently,  the resultant $R(B_{\Lambda},B_{\Lambda}')$ of $B_{\Lambda}$ and $B_{\Lambda}'(T)$ is a  polynomial in  $K[\Lambda]$, which is not identically zero.
Let 
	\begin{align}\label{zeroresultant}
		\text{$\alpha_i\in K$, $1\le i\le t$, be the zeros of  the resultant $R(B_{\Lambda},B_{\Lambda}')$.}
	\end{align}  
	It is clear that  $B(\lambda,T)$  has no multiple factors in $K[\lambda][T]$ if $\lambda\ne\alpha_i$ for any $1\le i\le t$.
Therefore, it is clear that we need to consider those $\lambda$ with  $\lambda\ne\alpha_i$ for any $1\le i\le t$  and $\lambda\ne\gamma_j$ for any $1\le j\le s$.  Assuming such, let $B(T):=\lambda^{M_2}B(\lambda,T)$ as in \eqref{GB}.
Let $\beta:=u_1^bu_2^{-a}$ and define $D_{ \beta }( B)\in  K_{\bf g}[T]$ as in \eqref{Duexpression}. 
By Lemma \ref{coprimeD}, the polynomials $B$ and $D_{ \beta }( B)$ are coprime in $K_{\bf g}[T]$.
 Let $\tilde B\in K(\lambda )[Z,U]$ and $\tilde D_{ \beta }( B)$ be the homogenization of  $B  $ and  $D_{ \beta }( B)$ respectively.    Write $\beta=\beta_1/\beta_0$, where $\beta_0$ and $\beta_1$ are entire functions without common zeros. 
 Then by Proposition \ref{gcdn1}
\begin{align}\label{n1} N_{\gcd} (\tilde B(\beta_0,\beta_1), \tilde  D_{\beta}(B)(\beta_0,\beta_1),r)\le{\rm o}( T_{\mathbf g}(r))
\end{align}
since $\beta$ is not constant.
 On the other hand, from the proof of  \cite[Proposition 5.3]{GW22}, there exists a proper Zariski closed set $W $ of $\mathbb P^2(\mathbb C)$, independent of $\mathbf{g}$, such that, if  image of $\mathbf{g}$ is  contained in $W$, 
\begin{align}\label{n2}		N_{G(\mathbf{g})}(0,r)-N_{G(\mathbf{g})}^{(1)}(0,r)\le_{\exc} N_{\gcd} (\tilde B(\beta_0,\beta_1),\tilde D_{ \beta }( B)(\beta_0,\beta_1),r).
\end{align}
Furthermore, $W$ can be described in Theorem \ref{main_thm_n2}.
We conclude the proof of \eqref{multizerobddn2} by combining (\ref{n1}) and (\ref{n2}).

 We now proceed to prove \eqref{truncate1_n2}.    
 Let $G=\sum_{\mathbf i\in I_G}\alpha_{\mathbf i}{\mathbf x}^{\mathbf i}\in K[ x_0,x_1, x_2]$.  Since the hypersurface $[G=0]$ and the coordinate hyperplanes  in $\mathbb P^2$ are in weakly general position, may write  
\begin{align}\label{expressG}
G(\mathbf g) = \sum_{0\le i\le 2}\alpha_{\mathbf i_i}g_i^d+\sum_{\mathbf i\in I_G\setminus I }\alpha_{\mathbf i}{\mathbf g}^{\mathbf i} 
\end{align}
where $\alpha_{\mathbf i_i}\ne 0$ for $0\le i\le n$ and $I=\{ (d,0,0),(0,d,0), (0,0,d)\}$.

 Let's express $B(\Lambda, T)$ in the form: 
\begin{align}\label{expressBT}
B(\Lambda, T)=\sum_{i\in I_B} b_i(\Lambda) T^{i}\in K[\Lambda][T],
\end{align} 
where $b_i\ne 0$ if $i\in I_B$.
We define $J\subset  K[\Lambda]$ as the finite set containing all $b_i(\Lambda)$ for $i\in I_B$ and all of their proper subsums.
Set $\mathcal R:=\{  r\in K\,|\, h(r)=0 \text{ for some  }h\in J\}$. 
 It is crucial to note that the proof of Theorem \ref{main_thm_1GExc}   has already demonstrated that \eqref{truncate1_n2} holds if neither $G(\mathbf g)$ nor any proper subsum of \eqref{expressG} is zero. 
Therefore, when evaluating $B(\Lambda, T)$ at $\Lambda=\lambda\notin \mathcal R$ and $T=\beta$, we need to consider equations of the following type:
\begin{align}\label{betaequation}
\sum_{i\in I_B} a_i(\lambda) \beta^{i}=0,
\end{align}
where $a_i(\Lambda)$ is a subsum of $b_i(\Lambda)$,  and there are at least two nontrivial  $ a_i$ in the left hand side of \eqref{betaequation} since $\lambda\notin \mathcal R$.
  Hence,  	$$
	T_{\beta}(r)\le c_3 T_{\lambda}(r)={\rm o}( T_{\mathbf{g}}(r)).$$
 This, however, leads to a contradiction.


\end{proof}

\end{document}